\documentclass[10pt]{amsart}
\usepackage{amssymb,latexsym,graphics}
\def\1a{{\mathbf 1}}
\def\2a{{\mathbf 2}}
\def\3a{{\mathbf 3}}
\def\4a{{\mathbf 4}}
\def\5a{{\mathbf 5}}
\def\6a{{\mathbf 6}}

\AtBeginDocument{

}

\numberwithin{equation}{section}
\newtheorem{theorem}{Theorem}[section]

\newtheorem{corollary}[theorem]{Corollary}

\newtheorem{lemma}[theorem]{Lemma}

\newtheorem{example}[theorem]{Example}

\newcommand{\Sch}{\mathcal{S}}
\newcommand{\D}{\mathcal{D}}
\newcommand{\ee}{\mathsf{e}}
\newcommand{\dd}{\mathsf{d}}
\newcommand{\uu}{\mathsf{u}}
\newcommand{\rr}{\mathsf{r}}

\newcommand\GETOUT[1]{}
\newcommand{\s}{\mathfrak{S}}
\newcommand{\C}{\mathcal{C}}

\title[Combinatorial Gray codes for classes of permutations]{Combinatorial Gray codes for classes of pattern avoiding permutations}
\author{W.M.B. Dukes, M.F. Flanagan, T. Mansour and V. Vajnovszki}
\address{Science Institute, University of Iceland, Reykjav\'{i}k, Iceland.}
\email{dukes@raunvis.hi.is}
\address{Institute for Digital Communications, The University of Edinburgh, The King's Buildings, Mayfield Road, Edinburgh EH9 3JL, Scotland.}
\email{mark.flanagan@ieee.org}
\address{Department of Mathematics, University of Haifa, 31905 Haifa, Israel.}
\email{toufik@math.haifa.ac.il}
\address{LE2I UMR CNRS 5158, Universit\'e de Bourgogne B.P. 47 870, 21078 DIJON-Cedex France}
\email{vvajnov@u-bourgogne.fr}
\begin{document}
\maketitle

\begin{abstract}
The past decade has seen a flurry of research into pattern avoiding
permutations but little of it is concerned with their exhaustive
generation. Many applications call for exhaustive generation of
permutations subject to various constraints or imposing a particular
generating order. In this paper we present generating algorithms and
combinatorial Gray codes for several families of pattern avoiding
permutations. Among the families under consideration are those
counted by Catalan, large Schr\"oder, Pell, even-index Fibonacci numbers
and the central binomial coefficients.
We thus provide
Gray codes for the set of all permutations of $\{1,\ldots , n\}$ avoiding
the pattern $\tau$
for all $\tau\in \s_3$ and
the
Gray codes we obtain have distances 4 or 5.
\end{abstract}

\section{Introduction}

A number of authors have been interested in Gray codes and generating
algorithms for permutations and their restrictions
(unrestricted \cite{Ehr_73}, with given {\it ups} and {\it downs} \cite{Kor_01,Roe_92},
involutions, and fixed-point free involutions \cite{Wal_01},
derangements \cite{baril_vaj}, permutations with a fixed number of cycles \cite{baril})
or their generalizations (multiset permutations \cite{Ko_92,Vaj_02_2}).
A recent paper \cite{juarna_vaj} presented Gray codes and generating algorithms
for the three classes of pattern avoiding permutations: $\s_n(123,132)$,
$\s_n(123,132,p(p-1)\ldots 1 (p+1))$,
and permutations in $\s_{n} (123,132)$ which have exactly ${\binom{n}{2}-k}$ inversions.
In \cite{bernini} a general technique
is presented for the generation of Gray codes for a large class of
combinatorial families; it is based on the ECO method and produces objects by their
encoding given by the generating tree (in some cases the obtained
encodings can be translated into the objects).
Motivated by these papers,
we investigate the related problem for several new classes of pattern
avoiding permutations.

More specifically, we give combinatorial Gray codes for classes
of pattern avoiding permutations which are counted by Catalan,
Schr\"oder, Pell, even-index Fibonacci numbers and the central binomial coefficients;
the Gray codes we obtain have distances 4 or 5.
Our work is different from similar work for combinatorial classes having the
same counting sequence, see for instance \cite{bernini,Vaj_02_1}.
%
%recent work of Bernini et. al.~\cite{bernini} which concentrates
%on providing a Gray code for the {\em{encoding}} of objects that are enumerated by the Catalan numbers.
%
Indeed, as Savage~\cite[\S 7]{savage} points out: `{\it{Since bijections are known between most members
of the Catalan family, a Gray code for one member of the family gives implicitly a listing scheme
for every other member of the family. However, the resulting list may not look like Gray codes,
since bijections need not preserve minimal changes between elements.}}'

Some direct constructions for $\s_n(231)$ exist but are, however, not Gray codes. For example,
B\'{o}na \cite[\S 8.1.2]{bonabook}
provides an algorithm for generating $\s_n(231)$.
This algorithm is such that the successor of the
permutation $\pi=(n,n-1,\ldots 2,1,    2n+1,2n,2n-1,\ldots , n+2,n+1)$ is
$\pi'=(1,2,\ldots, n-1, 2n+1,n,n+1,\ldots ,2n)$.
The number of places in which these two permutations differ is linear in $n$.

In Section 2 we present a combinatorial Gray code for $\s_n(231)$ with distance  4.
In Section 3 we present a Gray code for the Schr\"{o}der permutations, $\s_n(1243, 2143)$,
  with distance 5.
In Section 4 we present a general generating algorithm and Gray codes
for some classes of pattern avoiding permutations and discuss its limits.

The techniques we will use are:
in Section 2 and 3 reversing sublists  \cite{Rus_93};
in Section 3 combinatorial bijections \cite{juarna_vaj};
and in Section 4 generating trees \cite{bernini}.
%Note that this last method, unlike the previous work in \cite{BBGP, bernini},
%is used here to produce a Gray code for combinatorial objects
%in `natural' representation, and not the encoding
%given by the generating tree.

Throughout this paper, it is convenient to use the following
notation. The number $c_n=\frac{1}{n+1}\binom{2n}{n}$ is the $n$-th
Catalan number. The large Schr\"oder numbers $r_{n}$ are defined by $r_0=1$ and
for all $n > 0$,
\begin{eqnarray}
r_{n} &=& r_{n-1} + \sum_{k=1}^{n} r_{k-1} r_{n-k}.
\end{eqnarray}
Let $A(1)=0$, $B(1)=0$ and for all $i>1$,
\begin{eqnarray}
A(i)&=& c_0+\ldots + c_{i-2}, \mbox{ and } \label{eq_1_2}
\end{eqnarray}
\begin{eqnarray}
B(i)&=& r_0+\ldots + r_{i-2}. \label{eq_1_3}
\end{eqnarray}
The parity of these numbers will be extremely important in proving the
Gray code properties of the generating algorithms for permutations we define later on in the paper.
However, the parity of $A(i)$ and $B(i)$ are not explicitly used in the algorithms.
Note that for all $0<k\leq 2^n$, $A(2^n+k)$ is odd iff $n$ is even.
One can easily show that $B(i)$ is odd iff $i=2$.
%We retain the above summation expressions as they are crucial in the proofs.
For two permutations $\sigma=\sigma_1\sigma_2\ldots\sigma_n$
and  $\tau=\tau_1\tau_2\ldots\tau_n$ in $\s_n$, the metric $d(\sigma,\tau)$ is
the number of places in which they differ; and
we denote by $\sigma\circ\tau$ (or more compactly as $\sigma\tau$) their product, that is,
the permutation $\pi$ in $\s_n$ with $\pi_i=\tau_{\sigma_i}$ for all $i$,
$1\leq i\leq n$.
In particular, when $\sigma$ is the transposition $(u,v)$,
then $(u,v)\circ\tau$ is the permutation $\pi$ with $\pi_i=\tau_i$ for all $i$,
except that $\pi_u=\tau_v$ and $\pi_v=\tau_u$.

\section{A Gray code for $\s_n(231)$}
Note that if $(\pi(1),\ldots,\pi(c_n))$ is an ordered list of elements
of $\s_n(231)$ such that $d(\pi(i),\pi(i+1))\leq 4$, then
the operations of reverse, complement and their composition provide
lists for $\s_n(132)$, $\s_n(213)$ and $\s_n(312)$, respectively, which preserve
the distance between two adjacent permutations.

\subsection{Generating 231-avoiding permutations}
First we introduce some general notation
concerning the list $\D_n$ that our algorithm will generate and then provide the necessary proofs
to show that $\D_n$ is the desired object.

For every $n\ge 0$, let $\D_{n}$ denote a list consisting of $c_{n}$
entries, each of which is some permutation of $\{1,\ldots , n\}$.
The $j$-th entry is denoted $\D_{n}\left(j\right)$. In order that we
may copy such a list, either in its natural or reversed order, we
define $\D_{n}^{i}$ to be $\D_{n}$ if $i$ is odd, and $\D_n$
reversed if $i$ is even, for every positive integer $i$. Thus
$\D_{n}^{i}(j)\;=\;\D_{n}^{i+1}(c_{n}+1-j)$ for all $1\leq j\leq
c_{n}$.

By $\D_{n}\left(j\right)+l$ we shall mean $\D_n(j)$ with every element incremented by the value $l$.
Concatenation of lists is defined in the usual way, concatenation of any permutation with the null
permutation yields the same permutation, i.e. $[\tau, \; \emptyset ] \,=\, [\emptyset, \; \tau ] \, = \, \tau$.

The list $\D_{n}$ is defined recursively as follows; $\D_{0}$ consists
of a single entry which contains the null permutation that we denote
as $\emptyset$. For any $n\ge1$,
\begin{eqnarray}\label{list_231}
\D_{n}&=&\bigoplus_{i=1}^{n}\bigoplus_{j=1}^{c_{i-1}}\bigoplus_{k=1}^{c_{n-i}}
        \left[
        \D_{i-1}^{n+i-1}(j) , n , \D_{n-i}^{j+A(i)+1}(k)+(i-1)\right] , \label{eq:Dndef}
\end{eqnarray}
where $A(i)$ is defined in Equation (\ref{eq_1_2}) and $\oplus$ denotes the concatenation operator, e.g.
$$\bigoplus_{i=1}^2 \bigoplus_{j=1}^{2} \left(f(i,j)\right)
\; = \; \left( f(1,1), \, f(1,2) ,\, f(2,1) ,\, f(2,2) \right).$$

\begin{lemma}\label{21}
The list $\D_{n}$ contains all $231$-avoiding
permutations exactly once.
\end{lemma}

\begin{proof}
Every permutation $\pi \in \s_{n}(231)$ may be decomposed as
$\pi=\tau n \sigma$, where $\tau \in \s_{i-1}(231)$ and $\sigma$ is
a 231-avoiding permutation on the set $\{i,\ldots ,n-1\}$ which is
order-isomorphic to a $\sigma ' \in \s_{n-i}$. In $\D_n$, $n$
assumes the positions $i=1,2,\ldots , n$. For each position $i$ of
$n$, $\tau$ runs through $\D_{i-1}$ alternately forwards and
backwards, forwards the last time. For each $\tau$, $\sigma$ runs
through $\D_{i-1}+(i-1)$ alternately forwards and backwards,
backwards the first time (see Table 1). The result follows by strong
induction on $n$.
\end{proof}

\begin{lemma}\label{22}
For all $n\geq 2$,
$$\D_{n}(1)\;=\; n 1 2 3 \cdots (n-1) \mbox{ and }
\D_{n}(c_{n})\;=\; 1 2 3 \cdots n.$$
\end{lemma}

\begin{proof}
The proof proceeds by induction on $n$.
%From equation~(\ref{eq:Dndef})
We have $\D_{0}=\emptyset$.
Assume the result holds for each $i=0,1,2,\ldots n-1$. Then by Equation~(\ref{eq:Dndef}),
$\D_{n}(1)$ corresponds to the expression with $i=1,j=1$ and $k=1$;
\begin{eqnarray*}
\D_{n}(1)
    & = & n \; \D_{n-1}^{1+A(1)+1}(1) \; = \; n\;\D_{n-1}^2(1) \; = \; n\;\D_{n-1}(c_{n-1}) \; = \;
     n  1 2  3 \cdots (n-1).
\end{eqnarray*}
The last entry $\D_{n}(c_{n})$ corresponds to the expression in
Equation~(\ref{eq:Dndef}) with $i=n,j=c_{i-1}$ and $k=c_{n-i}$;
\begin{eqnarray*}
\D_{n}(c_{n})   & = & \D_{n-1}^{2n-1}(c_{n-1}) \;  n\;
        \; =\; 1 2 3 \cdots n.
\end{eqnarray*}
\end{proof}

\begin{theorem}\label{23}
For each $q\in\left\{ 1,2,\ldots c_{n}-1\right\} $,
$\D_{n}\left(q\right)$ differs from its successor
$\D_{n}(q+1)$ by a rotation of two, three or four elements.
\end{theorem}

\begin{proof}
The proof proceeds by induction.
The result holds trivially for $n=1$ since $\D_1$ consists of a single permutation.
Assume the
result holds for $\D_{i}$ for each $i=1,2,\ldots n-1$. From Equation~(\ref{eq:Dndef}),
there are 3 cases:

\begin{enumerate}
\item[(i)] The current permutation corresponds to $\left(i;j;k=t\right)$ and the
next permutation corresponds to $\left(i;j;k=t+1\right)$, where $t\in\left\{ 1,2,\ldots c_{n-i}-1\right\} $.
Therefore \begin{eqnarray*}
\D_{n}(q) & = &
\D_{i-1}^{n+i-1}(j) \; n \; \D_{n-i}^{j+A(i)+1}(t)+(i-1)\\
\D_{n}(q+1) & = &
\D_{i-1}^{n+i-1}(j) \; n \; \D_{n-i}^{j+A(i)+1}(t+1)+(i-1),
\end{eqnarray*}
and by the induction hypothesis, $$d(\D_n(q),\D_n(q+1)) =
d(\D_{n-i}(t),\D_{n-i}(t+1)) \leq 4.$$
\item[(ii)] The current permutation corresponds to $\left(i,j=t,k=c_{n-i}\right)$
and the next permutation corresponds to $\left(i;j=t+1;k=1\right)$,
where $t\in\left\{ 1,2,\ldots c_{i-1}-1\right\} $. Therefore
\begin{eqnarray*}
\D_{n}\left(q\right) & = &  \D_{i-1}^{n+i-1}(t) \; n \; \D_{n-i}^{t+A(i)+1}(c_{n-i})+(i-1) \\
\D_{n}\left(q+1\right) & = & \D_{i-1}^{n+i-1}(t+1) \; n \; \D_{n-i}^{t+A(i)+2}(1)+(i-1).
\end{eqnarray*}
Since $\D_{n-i}^{t+A(i)+1}(c_{n-i})=\D_{n-i}^{t+A(i)+2}(1)$,
the induction hypothesis gives
$$d(\D_n(q),\D_n(q+1)) = d(\D_{i-1}(t),\D_{i-1}(t+1)) \leq 4.$$
%**************************************************************

\item[(iii)] The current permutation corresponds to $\left(i=t;j=c_{i-1};k=c_{n-i}\right)$
and the next permutation corresponds to
$\left(i=t+1;j=1;k=1\right)$, where $t\in\{ 1,\ldots$ $n-1\}$.
Therefore
\begin{eqnarray*}
\D_{n}\left(q\right) & = & \D_{t-1}^{n+t-1}(c_{t-1}) \; n \; \D_{n-t}^{c_{t-1}+A(t)+1}(c_{n-t})+(t-1)\\
\D_{n}\left(q+1\right) & = & \D_{t}^{n+t}(1) \; n \; \D_{n-t-1}^{1+A(t+1)+1}(1)+t.
\end{eqnarray*}
This divides into four cases, where in each case we use Lemma~\ref{22}
and the fact that $A(t+1)=A(t)+c_{t-1}$:\\
%**************************************************************
(a)
If $n+t$ is odd and  $c_{t-1}+A(t)+1=A(t+1)+1$ is odd, then
\begin{eqnarray*}
\D_{n}(q)&=& 1 \, 2 \, 3 \, \ldots \, (t-1) \, n \, t \, (t+1) \, \ldots \, (n-1) \\
\D_{n}(q+1)&=& 1 \, 2 \, 3 \, \ldots \, (t-1) \, t \, n \, (t+1) \, \ldots \, (n-1).
\end{eqnarray*}
Here $\D_{n}(q+1)$ is obtained from $\D_{n}(q)$
via a single transposition of elements at positions $(t,t+1)$.\\
(b)
If $n+t$ is odd and $c_{t-1}+A(t)+1$ is even, then
\begin{eqnarray*}
\D_{n}(q)&=&   1 \, 2 \, \ldots \, (t-1) \, n \, (n-1) \, t \, (t+1) \, \ldots \, (n-2)\\
\D_{n}(q+1)&=& 1 \, 2 \, \ldots \, (t-1) \, t \, n \, (n-1) \, (t+1) \, \ldots \, (n-2),
\end{eqnarray*}
for all $t\leq n-3$.
Here $\D_{n}(q+1)$ is obtained from $\D_{n}(q)$
via a rotation of the 3 elements at positions $(t,t+1,t+2)$.
If $t=n-2$ then
\begin{eqnarray*}
\D_{n}(q)&=&   1 \, 2 \, \ldots \, (n-3) \, n \, (n-1) \, (n-2) \, \mbox{ and }\\
\D_{n}(q+1)&=& 1 \, 2 \, \ldots \, (n-3) \, (n-2) \, n \, (n-1).
\end{eqnarray*}
These permutations differ by a rotation of the 3 elements at positions $(n-2,n-1,n)$. If $t=n-1$ then
\begin{eqnarray*}
\D_{n}(q)&=&   (n-2)\, 1 \, 2 \, \ldots \, (n-3) \, n    \, (n-1) \, \mbox{ and }\\
\D_{n}(q+1)&=& (n-1)\, 1 \, 2 \, \ldots \, (n-3) \, (n-2) \, n .
\end{eqnarray*}
These permutations differ by a rotation of the 3 elements at positions $(1,n-1,n)$.\\
(c) If $n+t$ is even and $c_{t-1}+A(t)+1$ is odd, then
\begin{eqnarray*}
\D_{n}(q)&=& (t-1) \, 1 \, 2 \, \ldots \, (t-2) \, n \, t \, (t+1) \, \ldots \, (n-1)\;\mbox{ and }\\
\D_{n}(q+1)&=& t \, 1 \, 2 \, \ldots \, (t-2) \, (t-1) \, n \, (t+1) \,\ldots \, (n-1)
\end{eqnarray*}
for all $t\geq 3$.
Here $\D_{n}\left(q+1\right)$ is obtained from $\D_{n}(q)$
via a rotation of the 3 elements at positions $(1,t,t+1)$.\\
The degenerate cases $t=1,2$ are dealt with in the same manner as those at the end of part (b).\\
(d)
If $n+t$ is even and $c_{t-1}+A(t)+1$ is even, then
\begin{eqnarray*}
\D_{n}(q)&=& (t-1) \, 1 \, 2 \, \ldots \, (t-2) \, n \, (n-1) \, t \, (t+1) \, \ldots \, (n-2)\\
\D_{n}(q+1)&=& t \, 1 \, 2 \, \ldots \, (t-2) \, (t-1) \, n \, (n-1) \, (t+1) \, \ldots \, (n-2),
\end{eqnarray*}
for all $3\leq t \leq n-3$.
Here $\D_{n}\left(q+1\right)$ is obtained from $\D_{n}(q)$
via a rotation of the 4 elements at positions $(1,t,t+1,t+2)$.
The degenerate cases $t=1,2,n-2,n-1$ are dealt with in the same manner as those at the end of part (b).
\end{enumerate}
\end{proof}

In Table \ref{ttable_1} is given the list $\D_{6}$
obtained by relation (\ref{list_231}).
The alert reader will note that there is no rotation of 4 elements in Table 1.
Such a rotation is first observed when $n=7$ and $t=3$
(the permutation 2176345 becomes 3127645).

\begin{table}
\begin{center}
\caption{\label{ttable_1}
The Gray code $\D_{6}$ for the set $\s_6(231)$ given by
relation (\ref{list_231}) and produced by Algorithm 1.
Permutations are listed column-wise and
changed entries are in bold. }

{\small
 $\begin{array}{|c|c|c|c|c|c|}
 \hline
\begin{array}[t]{c}
6     1     2     3     4     5\\
6     \2a    \1a    3     4     5\\
6     \1a    \3a    \2a    4     5\\
6     \3a    \2a    \1a    4     5\\
6     3     \1a    \2a    4     5\\
6     \2a    1     \4a    \3a    5\\
6     \1a    \2a    4     3     5\\
6     1     \4a    \2a    3     5\\
6     1     4     \3a    \2a    5\\
6     \4a    \3a    \1a    2     5\\
6     4     3     \2a    \1a    5\\
6     4     \1a    \3a    \2a    5\\
6     4     \2a     \1a    \3a    5\\
6     4     \1a     \2a    3     5\\
6     \3a    1     2     \5a    \4a\\
6     3     \2a     \1a    5     4\\
6     \1a    \3a     \2a    5     4\\
6     \2a    \1a     \3a    5     4\\
6     \1a    \2a     3     5     4\\
6     1     2     \5a    \3a    4\\
6     1     2     5     \4a    \3a\\
6     \2a    \1a     5     4     3\\
\end{array} &
\begin{array}[t]{c}
6     2     1     5     \3a     \4a\\
6     \1a    \5a    \2a    3     4\\
6     1     5     \3a    \2a    4\\
6     1     5     \2a    \4a    \3a\\
6     1     5     \4a    \3a    \2a\\
6     1     5     4     \2a    \3a\\
6     \5a    \4a    \1a    2     3\\
6     5     4     \2a    \1a    3\\
6     5     4     \1a    \3a    \2a\\
6     5     4     \3a    \2a    \1a\\
6     5     4     3     \1a    \2a\\
6     5     \1a    \4a    \3a    2\\
6     5     1     4     \2a    \3a\\
6     5     1     \2a    \4a    3\\
6     5     \2a    \1a    4     3\\
6     5     \3a    1     \2a    \4a\\
6     5     3     \2a    \1a    4\\
6     5     \1a    \3a    \2a    4\\
6     5     \2a    \1a    \3a    4\\
6     5     \1a    \2a    3     4\\
\1a    \6a    \5a    2     3     4\\
1     6     5     \3a    \2a    4\\
\end{array} &
\begin{array}[t]{c}
1     6     5     \2a     \4a     \3a\\
1     6     5     \4a    \3a    \2a\\
1     6     5     4     \2a    \3a\\
1     6     \2a    \5a    \4a    3\\
1     6     2     5     \3a    \4a\\
1     6     2     \3a    \5a    4\\
1     6     \3a    \2a    5     4\\
1     6     \4a    2     \3a    \5a\\
1     6     4     \3a    \2a    5\\
1     6     \2a    \4a    \3a    5\\
1     6     \3a    \2a    \4a    5\\
1     6     \2a    \3a    4     5\\
1     \2a    \6a    3     4     5\\
1     2     6     \4a    \3a    5\\
1     2     6     \3a    \5a    \4a\\
1     2     6     \5a    \4a    \3a\\
1     2     6     5     \3a    \4a\\
\2a    \1a    6     5     3     4\\
2     1     6     5     \4a    \3a\\
2     1     6     \3a    \5a    \4a\\
2     1     6     \4a    \3a    5\\
2     1     6     \3a    \4a    5\\
\end{array} &
\begin{array}[t]{c}
\3a     1     \2a    \6a     4     5\\
3     1     2     6     \5a    \4a\\
3     \2a    \1a    6     5     4\\
3     2     1     6     \4a    \5a\\
\1a    \3a    \2a    6     4     5\\
1     3     2     6     \5a    \4a\\
\2a    \1a    \3a    6     5     4\\
2     1     3     6     \4a    \5a\\
\1a    \2a    3     6     4     5\\
1     2     3     6     \5a    \4a\\
1     2     3     \4a    \6a    \5a\\
\2a    \1a    3     4     6     5\\
\1a    \3a    \2a    4     6     5\\
\3a    \2a    \1a    4     6     5\\
3     \1a    \2a    4     6     5\\
\2a    1     \4a    \3a    6     5\\
\1a    \2a    4     3     6     5\\
1     \4a    \2a    3     6     5\\
1     4     \3a    \2a    6     5\\
\4a    \3a    \1a    2     6     5\\
4     3     \2a    \1a    6     5\\
4     \1a    \3a    \2a    6     5\\
\end{array} &
\begin{array}[t]{c}
4     \2a     \1a     \3a     6     5\\
4     \1a    \2a    3     6     5\\
\5a     1    2     3     \4a    \6a\\
5     \2a    \1a    3     4     6\\
5     \1a    \3a    \2a    4     6\\
5     \3a    \2a    \1a    4     6\\
5     3     \1a    \2a    4     6\\
5     \2a    1     \4a    \3a    6\\
5     \1a    \2a    4     3     6\\
5     1     \4a    \2a    3     6\\
5     1     4     \3a    \2a    6\\
5     \4a    \3a    \1a    2     6\\
5     4     3     \2a    \1a    6\\
5     4     \1a    \3a    \2a    6\\
5     4     \2a    \1a    \3a    6\\
5     4     \1a    \2a    3     6\\
\1a    \5a    \4a    2     3     6\\
1     5     4     \3a    \2a    6\\
1     5     \2a    \4a    \3a    6\\
1     5     \3a    \2a    \4a    6\\
1     5     \2a    \3a    4     6\\
\2a    \1a    \5a    3     4     6\\
\end{array} &
\begin{array}[t]{c}
2     1     5     \4a     \3a     6\\
\1a    \2a    5     4     3     6\\
1     2     5     \3a    \4a     6\\
1     2     \3a    \5a    4     6\\
\2a    \1a    3     5     4     6\\
\1a    \3a    \2a    5     4     6\\
\3a    \2a    \1a    5     4     6\\
3     \1a    \2a    5     4     6\\
\4a    1     2     \3a    \5a    6\\
4     \2a    \1a    3     5     6\\
4     \1a    \3a    \2a    5     6\\
4     \3a    \2a    \1a    5     6\\
4     3     \1a    \2a    5     6\\
\1a    \4a    \3a    2     5     6\\
1     4     \2a    \3a    5     6\\
1     \2a    \4a    3     5     6\\
\2a    \1a    4     3     5     6\\
\3a    1     \2a    \4a    5     6\\
3     \2a    \1a    4     5     6\\
\1a    \3a    \2a    4     5     6\\
\2a    \1a    \3a    4     5     6\\
\1a    \2a    3     4     5     6\\
\end{array}\\
\hline
\end{array}$
}
\end{center}
\end{table}

\begin{figure}
\begin{tabular}{l}
  \hline
  {\bf Algorithm 1} Pseudocode for generating $\s_N(231)$ using Equation (\ref{list_231}).
  The list\\
  $\D_{n}$ is computed for each $1\leq n \leq N$. Here $\D_{n}^{R}$ denotes the reversal of list $\D_{n}$.\\\hline
    $\quad$\textbf{set} $D_{0}$ to a $1\times0$ matrix\\
    $\quad$\textbf{set} $D_{1}:=[1]$\\
    $\quad$\textbf{for} $n:=2$ {\bf to} $N$ {\bf do}\\
    $\qquad\tau\textrm{state}:=n\,\,\,\,(\textrm{mod }2)\qquad$ \{1 means forwards and 0 means backwards\}\\
    $\qquad\sigma\textrm{state}:=0$ \\
    $\qquad${\bf for} $i:=1$ {\bf to} $n$ {\bf do}\\
    $\qquad\quad${\bf for} $l:=1$ {\bf to} $i-1$ {\bf do}\\
    $\qquad\qquad${\bf if} $\tau\textrm{state}=0$ {\bf then}\\
    $\qquad\qquad\quad\tau=:D_{i-1}^{R}\left(l\right)$\\
    $\qquad\qquad${\bf else}\\
    $\qquad\qquad\quad\tau:=D_{i-1}\left(l\right)$\\
    $\qquad\qquad${\bf end if}\\
    $\qquad\qquad${\bf for} $r:=1$ {\bf to} $c_{n-i}$ {\bf do}\\
    $\qquad\qquad\quad${\bf if} $\sigma$\textrm{state}=0 {\bf then}\\
    $\qquad\qquad\qquad\sigma:=D_{n-i}^{R}\left(r\right)+\left(i-1\right)$\\
    $\qquad\qquad\quad${\bf else}\\
    $\qquad\qquad\qquad\sigma:=D_{n-i}\left(r\right)+\left(i-1\right)$\\
    $\qquad\qquad\quad${\bf end if}\\
    $\qquad\qquad\quad$new\_row$:=\left[\tau , n , \sigma\right]$\\
    $\qquad\qquad\quad$Append new\_row to $D_{n}$\\
    $\qquad\qquad${\bf end for}\\
    $\qquad\qquad\sigma\textrm{state}:=\sigma\textrm{state}+1\,\,\,\,(\textrm{mod }2)$\\
    $\qquad\quad${\bf end for}\\
    $\qquad\quad\tau\textrm{state}:=\tau\textrm{state}+1\,\,\,\,(\textrm{mod }2)$\\
    $\qquad${\bf end for}\\
    $\quad${\bf end for}\\\hline
\end{tabular}
\end{figure}

\section{A Gray code for Schr\"oder permutations}
The permutations $\s_n(1243,2143)$ are called Schr\"oder permutations and are just one of the classes of permutations
enumerated by the Schr\"oder numbers mentioned in the Introduction.
Let $\mathcal{S}_n$ be the class of Schr\"{o}der paths from (0,0) to $(2n,0)$
(such paths may take steps $\uu =(1,1)$, $\dd =(1,-1)$ and $\ee =(2,0)$ but never go below the $x$-axis).
This class $\mathcal{S}_n$ is enumerated by $r_n$, see for instance \cite{egge_mansour}.

In what follows, we will present a recursive procedure for
generating all Schr\"{o}der paths of length $n$. This procedure has
the property that if the paths in $\mathcal{S}_n$ are listed as
$(p_1,p_2,\ldots)$, then the sequence of permutations
$(\varphi(p_1),\varphi (p_2),\ldots)$ is a Gray code for
$\s_{n+1}(1243,2143)$ with distance 5. First we briefly describe
Egge and Mansour's~\cite [\S 4]{egge_mansour} bijection $\varphi:
\mathcal{S}_n \mapsto \s_{n+1}(1243,2143)$.

Let $p\in \mathcal{S}_n$ and let $s_i$ be the transposition $(i,i+1)$.
\begin{description}
\item[Step 1]
For all integers $a$, $m$ with $0\leq a,m <n$, if either of the
points $((8m+1)/4,(8a+5)/4)$ or $((8m+5)/4,(8a+1)/4)$ is contained
in the region beneath $p$ and above the $x$-axis, then place a
dot at that point.
For such a dot, with coordinates $(x,y)$, associate the label $s_i$ where $i=(1+x-y)/2$.
Let $j=1$.
\item[Step 2]
Choose the rightmost dot that has no line associated with it (with label $s_k$, say).
Draw a line parallel to the $x$-axis from this dot to the leftmost dot that may be
reached without crossing $p$ (which has label $s_l$, say).
Let $\sigma_j = s_k s_{k-1} \ldots  s_l$, where $s_i$, applied to a permutation $\pi$, exchanges $\pi_i$ with $\pi_{i+1}$.
If all dots have lines running through them, then go to
step 3. Otherwise increase $j$ by $1$ and repeat step 2.
\item[Step 3]
Let $\varphi(p) = \sigma_j \ldots \sigma_2 \sigma_1 (n+1,n,\ldots , 1)$.
\end{description}

\begin{example}
Consider the path $p \in \mathcal{S}_6$ in the diagram.\\[0.5em]
\centerline{\scalebox{0.8}{\includegraphics{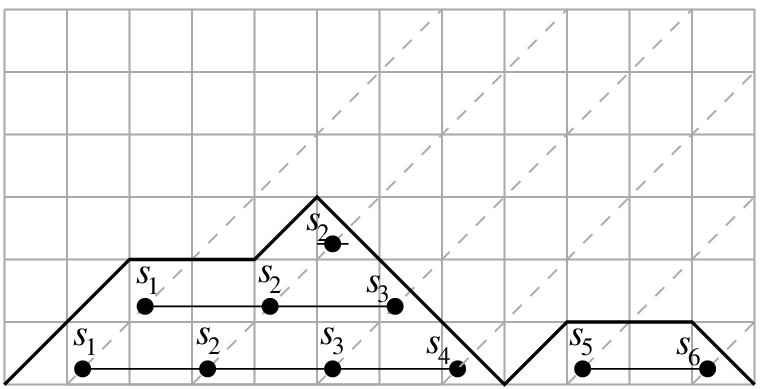}}}
The dots indicate the points realized in Step 1 and the lines
joining them indicate how each of the $\sigma$'s are formed. We have
$\sigma_1=s_6s_5$, $\sigma_2=s_4s_3s_2s_1$, $\sigma_3= s_3s_2s_1$
and $\sigma_4=s_2$. So
\begin{eqnarray*}
\varphi(p) &=& \sigma_4 \sigma_3 \sigma_2 \sigma_1 (7,6,5,4,3,2,1)\\
&=& s_2 \; s_3s_2s_1 \; s_4s_3s_2s_1 \; s_6s_5 (7,6,5,4,3,2,1)  \\
&=& (5,2,4,6,7,1,3).
\end{eqnarray*}
\end{example}

\subsection{Generating all Schr\"{o}der paths}
There are many ways to recursively generate all Schr\"{o}der paths of length $n$.
In what follows, we give one such procedure for generating the list $\Sch_n$.
This list has the property that
the corresponding permutations, under the bijection $\varphi$, are
a Gray code for Schr\"oder permutations of distance 5.

As in Section 2, we will use the convention that for any integer $i$,
$\mathcal{S}_n^i = \mathcal{S}_n$ if $i$ is odd and
$\mathcal{S}_n^i$ is $\mathcal{S}_n$ reversed, if $i$ is even.
Entry $j$ of $\mathcal{S}_n$ is denoted $\mathcal{S}_n(j)$.
In this notation we will have
\begin{eqnarray*}\label{Sch_list}
\mathcal{S}_n^i(j) &=& \left\{
        \begin{array}{ll}
        \mathcal{S}_n(j) & \mbox{ if $i$ is odd,}\\
        \mathcal{S}_n(r_n+1-j) & \mbox{ if $i$ is even.}\\
        \end{array}
        \right.
\end{eqnarray*}

Define $\Sch_0$ to be the list consisting of the single null Schr\"oder path,
denoted $\emptyset$.
For all $n\geq 1$, the paths are generated recursively via
\begin{eqnarray} \label{schr_gen}
\Sch_{n} &=& \bigoplus_{i=1}^{r_{n-1}} \left(\ee\, \Sch_{n-1}(i)\right)
        \oplus
        \bigoplus_{i=1}^{n}
        \bigoplus_{j=1}^{r_{i-1}}
        \bigoplus_{k=1}^{r_{n-i}}
        \left(\uu \, \Sch^{n+i}_{i-1}(j)\, \dd\,  \Sch^{j+B(i)+1}_{n-i}(k)\right).
\end{eqnarray}
$\Sch_n$ starts with $\Sch_{n-1}$ with each path preceded by $\ee$.
There follow all the Schr\"oder paths beginning with $\uu$. Let
$\dd$ be the partner of this $\uu$ (the $\dd$ that returns the path
to the $x$ axis). Then $\dd$ assumes positions $i=2,4,6,\ldots,2n$
in the path. For each $i$, we have the paths in $\uu\, \alpha\, \dd
\,\beta$, where $\alpha$ runs through $\Sch_{i-1}$ alternately
forwards and backwards, backwards the last time, and for each
$\alpha$, $\beta$ runs through $\Sch_{n-i}$ alternately forwards and
backwards, backwards the first time.

Furthermore, we define $\Phi_n(j) := \varphi(\Sch_n(j))$ and
\begin{eqnarray}
\Phi_n &:=& \bigoplus_{j=1}^{r_n} \Phi_n(j).
\end{eqnarray}
For example, we have $\Sch_1=(\ee , \uu\dd)$ and $\Sch_2=(\ee\ee ,
\ee\uu\dd , \uu\dd\uu\dd , \uu\dd\ee , \uu\uu\dd\dd , \uu\ee\dd )$.
Thus $\Phi_1 = (21,12)$ and $\Phi_2=(321,312, 132, 231,123,213)$.
The paths and permutations $\Sch_3$, $\Phi_3$, $\Sch_4$ and $\Phi_4$
are listed in Tables 2 and 3. For two paths $p_1,p_2 \in \Sch_n$, we
write $d(p_1,p_2)$ for the number of places in which the two paths
differ when each $\ee$ is replaced by $\rr\rr$ where $\rr$
represents (1,0); e.g. $d(\ee , \uu\dd )=2$ and $d(\uu\ee\dd
,\ee\uu\dd )=2$.

\begin{lemma}\label{31}
Equation (\ref{schr_gen}) generates all Schr\"oder paths of length $n$.
\end{lemma}

\begin{proof}
This is routine by induction. The first concatenation operator
forms all paths that begin with step $\ee$.
If a path does not begin with $\ee$, then it does not touch the
$x$ axis for the first time until $(2i,0)$.
A path of this form is uniquely expressed as
$\uu \alpha \dd \beta$ where $\alpha \in \Sch_{i-1}$ and $\beta \in \Sch_{n-i}$.
\end{proof}

\begin{lemma} \label{firstlast}
For all $n\geq1$, $\mathcal{S}_n(1)=\ee ^{n}$ and $\mathcal{S}_n(r_n)=\uu \ee^{n-1}\dd$.
\end{lemma}

\begin{proof}
By Equation (\ref{schr_gen}) we have that $\mathcal{S}_1(1)\,=\,\ee$
and $\mathcal{S}_1(2) \,=\, \uu\dd$; so the result is true for
$n=1$. Assume it to be true for all $m\leq n-1$. Then $S_n(1) \,=\,
\ee \, \Sch_{n-1}(1) \,=\, \ee  \, \ee^{n-1} \,=\, \ee^n$.

Similarly, $\Sch_n(r_n)$ corresponds to Equation (\ref{schr_gen}) with $i=n,j=r_{n-1},k=r_{0}$, thus
\begin{eqnarray*}
\Sch_n(r_n) &=& \uu\,\Sch_{n-1}^{2n} (r_{n-1})\, \dd \;=\;\uu\, \ee^{n-1}\, \dd \; = \; \uu\, \ee^{n-1}\,\dd.
\end{eqnarray*}
Hence by induction the result is true for all $n\geq 1$.
\end{proof}

Under the bijection $\varphi$, we thus have

\begin{corollary}
For all $n>0$,
\begin{eqnarray*}
\Phi_n(1) &=& (n+1)\, n\, \ldots \, 1,\\
\Phi_n(r_n) &=& n\,\ldots \,1\,(n+1).
\end{eqnarray*}
\end{corollary}

\begin{theorem}\label{th_Sch_Path}
For each $1\leq  q < r_n$, $\Sch_n(q)$ differs from $\Sch_n(q+1)$ in at most
5 places and $d(\Phi_n(q),\Phi_n(q+1))\leq 5$.
\end{theorem}

\begin{proof}
This proof follows by strong induction and analyzing the different
successors that occur in Equation (\ref{schr_gen}). The statement in
the Theorem holds for $n=0$ because there is only one permutation.
We assume the statement in the Theorem holds true for all $0\leq i
\leq n-1$. From Equation (\ref{schr_gen}) there are five cases to
consider:
\begin{enumerate}
\item[(i)] If $1\leq q < r_{n-1}-1$, then
$\Sch_n(q) \,=\, \ee  \,\Sch_{n-1}(q)$ and $\Sch_n(q+1)\,=\,\ee \,\Sch_{n-1}(q+1)$.
This gives
\begin{eqnarray*}
d(\Sch_n(q),\Sch_n(q+1)) & = & d(\Sch_{n-1}(q), \Sch_{n-1}(q+1)),
\end{eqnarray*}
which is $\leq 5$ by our hypothesis.
Thus
\begin{eqnarray*}
\Phi_n(q) &=& (n+1)\,\Phi_{n-1}(q) \mbox{ and }\\
\Phi_n(q+1) &=& (n+1)\,\Phi_{n-1}(q+1) ,
\end{eqnarray*}
and so $d(\Phi_n(q),\Phi_n(q+1)) \leq 5$.
%========
\item[(ii)]
If $q=r_{n-1}$ then by Equation (\ref{schr_gen}) with $(i=1;j=1;k=1)$ and Lemma~\ref{firstlast} we have
\begin{eqnarray*}
\Sch_n(r_{n-1}) &=& \ee \, \Sch_{n-1}(r_{n-1})  \; = \; \ee\,\uu\,\ee^{n-2}\,\dd \mbox{ and }\\
\Sch_n(r_{n-1}+1) &=& \uu \, \dd  \, \Sch_{n-1}^2 (1)
    \; = \; \uu\,\dd\,\uu\,\ee^{n-2}\,\dd.
\end{eqnarray*}
Thus $d(\Sch_n(r_{n-1}),\Sch_n(r_{n-1}+1)) = d(\ee\uu\ee^{n-2}\dd, \uu\dd\uu\ee^{n-2}\dd) = 2$.
The corresponding permutations are
\begin{eqnarray*}
\Phi_n(r_{n-1}) &=& (n+1)\,(n-1)\,(n-2)\,\ldots\,2\,1\,n \mbox{ and }\\
\Phi_n(r_{n-1}+1) &=& (n-1)\,(n+1)\,(n-2)\,\ldots \, 2\, 1\, n,
\end{eqnarray*}
so that $d(\Phi_n(r_{n-1}),\Phi_n(r_{n-1}+1))\,=\, 2\, \leq 5$.
\item[(iii)]
If $\Sch_n(q)$ corresponds to $(i;j=r_{i-1};k=t)$ for some $1\leq t < r_{n-i}$ in
Equation (\ref{schr_gen}) then
\begin{eqnarray*}
\Sch_n(q)   &=& \uu \, \Sch_{i-1}^{n+i}(r_{i-1}) \, \dd \, \Sch_{n-i}^{j+B(i)+1}(t) \mbox{ and } \\
\Sch_n(q+1) &=& \uu \, \Sch_{i-1}^{n+i}(r_{i-1}) \, \dd \, \Sch_{n-i}^{j+B(i)+1}(t+1),
\end{eqnarray*}
and the distance of the two paths is no greater than 5, by the induction hypothesis.
Therefore
\begin{eqnarray*}
\Phi_n(q) &=& a\circ (n+1,\ldots,n+2-i,\varphi(\Sch_{n-i}^{j+B(i)+1}(t)))\mbox{ and } \\
\Phi_n(q+1) &=& a\circ (n+1,\ldots,n+2-i,\varphi(\Sch_{n-i}^{j+B(i)+1}(t+1))) ,
\end{eqnarray*}
where
\begin{eqnarray*}
a &=& \left\{
    \begin{array}{l@{\quad}l}
    s_i s_{i-1} \ldots s_1, & \mbox{if $n+i$ even},\\
    s_{i-1} \ldots s_1 s_i s_{i-1} \ldots s_1, & \mbox{if $n+i$ odd}.\\
    \end{array}
    \right.
\end{eqnarray*}
Using the fact that if $d(b,b')\leq x$, then $d(a\circ b,a\circ b')\leq x$, we have by the induction hypothesis
$d(\Phi_n(q),\Phi_n(q+1)) \leq 5$.

\item[(iv)]
If $\Sch_n(q)$ corresponds to Equation (\ref{schr_gen}) with triple
$(i;j=t;k=r_{n-i})$, where $1\leq t < r_{i-1}$, then the successor $\Sch_n(q+1)$ corresponds to
Equation (\ref{schr_gen}) with triple $(i;j=t+1;k=1)$. Consequently,
\begin{eqnarray*}
\Sch_n(q) &=& \uu  \, \Sch_{i-1}^{n+i}(t) \, \dd  \, \Sch_{n-i}^{t+B(i)+1}(r_{n-i}) \mbox{ and } \\
\Sch_n(q+1) &=& \uu \, \Sch_{i-1}^{n+i}(t+1) \, \dd  \, \Sch_{n-i}^{t+B(i)+2}(1).
\end{eqnarray*}
Since $\Sch_{n-i}^{t+B(i)+1}(r_{n-i}) = \Sch_{n-i}^{t+B(i)+2}(1)$, the result
for $\Sch_n$ follows by the induction hypothesis applied to $\Sch_{i-1}^{n+i}$.
Now if $t+B(i)+2$ is odd, then
\begin{eqnarray*}
\Phi_n(q) &=& \hat{\varphi}(\uu\,\Sch_{i-1}^{n+i}(t)\,\dd) \, i\,(i-1)\,\ldots\,1 \mbox{ and }\\
\Phi_n(q+1) &=& \hat{\varphi}(\uu\,\Sch_{i-1}^{n+i}(t+1)\,\dd) \,i\,(i-1)\,\ldots\,1 ,
\end{eqnarray*}
where $\hat{\varphi}(\uu\,\Sch_{i-1}^{n+i}(t)\,\dd)$ is
$\varphi(\uu\,\Sch_{i-1}^{n+i}(t)\,\dd)$ with every element
incremented by $i$. Since
$d(\Sch_{i-1}^{n+i}(t),\Sch_{i-1}^{n+i}(t+1)) \leq 5$, we have that
$d(\Phi_n(q),\Phi_n(q+1)) \leq 5$. The case where $t+B(i)+2$ is even
is handled in a similar manner with the suffix $i(i-1)\ldots 1$
replaced by $(i-1)\ldots 1 (i+1)$.

\item[(v)]
If $\Sch_n(q)$ corresponds to Equation (\ref{schr_gen}) with triple
$(i=t;j=r_{i-1};k=r_{n-i})$, where $1\leq t<n$, then $\Sch_n(q+1)$ corresponds to
Equation (\ref{schr_gen}) with triple $(i=t+1;j=1;k=1)$.
Consequently
\begin{eqnarray*}
\Sch_n(q) &=& \uu \, \Sch_{t-1}^{n+t}(r_{t-1}) \, \dd \, \Sch_{n-t}^{r_{t-1}+B(t)+1}(r_{n-t})  \mbox{ and }\\
\Sch_n(q+1) &=& \uu \, \Sch_{t}^{n+t+1}(1) \, \dd \, \Sch_{n-t-1}^{1+B(t+1)+1}(1).
\end{eqnarray*}
This divides into 4 sub-cases depending on the parity of the numbers $n+t$ and
$r_{t-1}+B(t)+1=B(t+1)+1$.
Each case is easily resolved by applying Lemma~\ref{firstlast}.
    \begin{enumerate}
    \item[(a)] If $n+t$ is even and $B(t+1)+1$ is even, then
    \begin{eqnarray*}
    \Sch_n(q) &=& \uu \, \Sch_{t-1}^2(r_{t-1}) \, \dd \, \Sch_{n-t}^2(r_{n-t}) \; = \; \uu \,\ee^{t-1} \, \dd \, \ee^{n-t} \mbox{ and } \\
    \Sch_n(q+1) &=& \uu \, \Sch_{t}(1) \, \dd \, \Sch_{n-t-1}(1) \; = \; \uu \, \ee^{t} \, \dd \, \ee^{n-t-1},
    \end{eqnarray*}
    which differ in two positions. This gives
    \begin{eqnarray*}
    \Phi_n(q) &=& n\,(n-1)\,\ldots\,(n-t+1)\,(n+1)\,(n-t)\,(n-t-1)\,\ldots \,1 \mbox{ and }\\
    \Phi_n(q+1) &=& n\,(n-1)\,\ldots\,(n-t)\,(n+1)\,(n-t-1)\,\ldots \,1,
    \end{eqnarray*}
    for all $1\leq t \leq n-1$. The two permutations differ by a transposing the elements at positions $(t+1,t+2)$.
    \item[(b)] If $n+t$ is odd and $B(t+1)+1$ is odd, then
    \begin{eqnarray*}
    \Sch_n(q) &=& \uu \, \Sch_{t-1}(r_{t-1}) \, \dd \, \Sch_{n-t}(r_{n-t}) \; =\; \uu \, \uu\ee^{t-2}\dd \, \dd \, \uu \ee^{n-t-1} \dd \mbox{ and }\\
    \Sch_n(q+1) &=& \uu \, \Sch_{t}^2(1) \, \dd \, \Sch_{n-t-1}^2(1) \; =\;  \uu  \, \uu\ee^{t-1}\dd \, \dd \, \uu \ee^{n-t-2} \dd ,
    \end{eqnarray*}
    which differ in five positions. This gives
    \begin{eqnarray*}
    \Phi_n(q) &=& (n-1)\cdots(n-t+2)(n-t)n(n+1)(n-t-1)\cdots 1(n-t+1) \mbox{ and }\\
    \Phi_n(q+1)&=& (n-1)\cdots(n-t+1)(n-t-1)n(n+1)(n-t-2)\cdots 1(n-t),
    \end{eqnarray*}
    for all $2\leq t \leq n-2$.
    These two permutations differ in five places (a transposition of the positions $(t-1,n)$ and a cycle of three elements at positions $(t ,t+1 ,t+2 )$).
For $t=1$ we have
    \begin{eqnarray*}
    \Phi_n(q) &=& n\,(n+1)\,(n-1)\,(n-2)\,\ldots \,1 \mbox{ and }\\
    \Phi_n(q+1) &=& (n-1)\, n \, (n+1)\, (n-2)\,\ldots \,1,
    \end{eqnarray*}
    which differ by a cycle of three elements at positions (1,2,3).
    Similarly, for $t=n-1$ we have
    \begin{eqnarray*}
    \Phi_n(q) &=& (n-1)\,\ldots \,1\, (n+1) \, n \, \mbox{ and }\\
    \Phi_n(q+1) &=& (n-1)\, \ldots \, 1 \, n \, (n+1),
    \end{eqnarray*}
    which differ by transposing the entries in positions $(n,n+1)$.
    \item[(c)] If $n+t$ is odd and $B(t+1)+1$ is even, then
    \begin{eqnarray*}
    \Sch_n(q) &=& \uu \, \Sch_{t-1}(r_{t-1}) \, \dd \, \Sch_{n-t}^2(r_{n-t}) \; = \; \uu \, \uu\ee^{t-2}\dd \, \dd \, \ee^{n-t} \mbox{ and } \\
    \Sch_n(q+1) &=& \uu \, \Sch_{t}^2(1) \, \dd \, \Sch_{n-t-1}(1) \; = \; \uu \, \uu\ee^{t-1}\dd \, \dd \,  \ee^{n-t-1}  .
    \end{eqnarray*}
    Thus $\Sch_n(q+1)$ differs from $\Sch_n(q)$ in four positions.
    This gives
    \begin{eqnarray*}
    \Phi_n(q) &=& (n-1)\,\ldots\,(n-t+1)\,n\,(n+1)\,(n-t)\,\ldots \,1\mbox{ and }\\
    \Phi_n(q+1) &=& (n-1)\,\ldots\,(n-t)\,n\,(n+1)\,(n-t-1)\,\ldots \,1,
    \end{eqnarray*}
    for all $t\geq 2$.
    The two permutations differ in three places (a rotation of three elements at positions $(t,t+1,t+2)$).
    The degenerate case $t=1$ is handled in the same manner as in part (a).
    \item[(d)] If $n+t$ is even and $B(t+1)+1$ is odd, then
    \begin{eqnarray*}
    \Sch_n(q) &=& \uu \, \Sch_{t-1}^2(r_{t-1}) \, \dd \, \Sch_{n-t}(r_{n-t}) \; =\; \uu \, \ee^{t-1} \, \dd \, \uu\ee^{n-t-1}\dd \mbox{ and } \\
    \Sch_n(q+1) &=& \uu \, \Sch_{t}(1) \, \dd \, \Sch_{n-t-1}^2(1) \; = \; \uu \, \ee^{t} \, \dd \,  \uu\,\ee^{n-t-2}\,\dd  .
    \end{eqnarray*}
    Thus $\Sch_n(q+1)$ differs from $\Sch_n(q)$ in five positions.
    This gives
    $$\begin{array}{l}
    \Phi_n(q)=n(n-1)\cdots(n-t+2)(n-t)(n+1)(n-t-1)\cdots 1(n-t+1)
    \end{array}$$
    and
    $$\begin{array}{l}
    \Phi_n(q+1)= n(n-1)\cdots(n-t+1)(n-t-1)(n+1)(n-t-2)\cdots 1(n-t),
    \end{array}$$
    for all $t\leq n-2$.
    The two permutations differ in four places (the two disjoint transpositions of elements at positions $(t,n+1)$ and $(t+1,t+2)$).
    The degenerate case $t=n-1$ is handled in the same manner as in part (a).
    \end{enumerate}
\end{enumerate}
\end{proof}

The lists $\Sch_3$, $\Phi_3$, $\Sch_4$ and $\Phi_4$ are given in Table
\ref{list_S_P_3} and \ref{list_S_P_4}.
Note that, unlike $\Phi_n$, the list $\Sch_n$ is a circular Gray
code; its first and last element have distance at most five.
The choice of a Gray code for Schr\"oder paths is critical in our
construction of a Gray code for $\s_n(1243,2143)$ since Egge and
Mansour's bijection $\varphi$, generally, does not preserves
distances. For instance $d(\ee^{n},\uu\ee^{n-1}\dd)=2$ but
$\varphi(\ee^{n})=(n+1)n\ldots 1$ differs from
$\varphi(\uu\ee^{n-1}\dd)=n\ldots 1(n+1)$ in all positions. Also,
there already exists a distance five Gray code for Schr\"oder paths
\cite{Vaj_02_1} but it is not transformed into a Gray code for
$\s_n(1243, 2143)$ by a known bijection. Finally, as in the previous
section, both Gray codes presented above can be implemented in
exhaustive generating algorithms.

\begin{table}
\begin{center}
\caption{\label{list_S_P_3}
The lists $\Sch_3$ and $\Phi_3$.}

{\small
 $\begin{array}{|c|c|c|} \hline
\begin{array}[t]{ccc} n & \Sch_3(n) & \Phi_3(n) \\ \hline
1 & \ee\ee\ee & 4321 \\
2 & \ee\ee\uu\dd & 4312 \\
3 & \ee\uu\dd\uu\dd & 4132 \\
4 & \ee\uu\dd\ee & 4231 \\
5 & \ee\uu\uu\dd\dd &4123 \\
6 & \ee\uu\ee\dd & 4213 \\
7 & \uu\dd \uu\ee\dd & 2413 \\
8 & \uu\dd \uu\uu\dd\dd &1423
\end{array} & \begin{array}[t]{ccc} n & \Sch_3(n) & \Phi_3(n) \\ \hline
9 & \uu\dd \uu\dd\ee &2431 \\
10 & \uu\dd \uu\dd\uu\dd &1432 \\
11 & \uu\dd \ee\uu\dd &3412 \\
12 & \uu\dd \ee\ee &3421 \\
13 & \uu\dd\ee \ee &3241 \\
14 & \uu\ee\dd \uu\dd &3142 \\
15 & \uu\uu\dd\dd \uu\dd &1342 \\
16 & \uu\uu\dd\dd \ee &2341
\end{array} & \begin{array}[t]{ccc} n & \Sch_3(n) & \Phi_3(n) \\ \hline
17 & \uu\uu\ee\dd \dd &2134 \\
18 & \uu \uu\uu\dd\dd \dd &1234 \\
19 & \uu \uu\dd\ee \dd &2314 \\
20 & \uu \uu\dd\uu\dd \dd & 1324 \\
21 & \uu \ee\uu\dd \dd &3124 \\
22 & \uu \ee\ee \dd & 3214
\end{array} \\ \hline
\end{array}$
}
\end{center}
\end{table}

\begin{table}
\begin{center}
\caption{\label{list_S_P_4}
The lists $\Sch_4$ and $\Phi_4$.}

{\small
$\begin{array}{|c|c|c|} \hline
\begin{array}[t]{ccc}
n & \Sch_4(n) & \Phi_4(n) \\ \hline
1 & \ee\ee\ee\ee & 54321 \\
2 & \ee\ee\ee\uu\dd & 54312 \\
3 & \ee\ee\uu\dd\uu\dd & 54132 \\
4 & \ee\ee\uu\dd\ee & 54231 \\
5 & \ee\ee\uu\uu\dd\dd & 54123 \\
6 & \ee\ee\uu\ee\dd & 54213 \\
7 & \ee\uu\dd\uu\ee\dd & 52413 \\
8 & \ee\uu\dd\uu\uu\dd\dd & 51423 \\
9 & \ee\uu\dd\uu\dd\ee & 52431 \\
10 & \ee\uu\dd\uu\dd\uu\dd & 51432 \\
11 & \ee\uu\dd\ee\uu\dd & 53412 \\
12 & \ee\uu\dd\ee\ee & 53421 \\
13 & \ee\uu\dd\ee\ee & 53241 \\
14 & \ee\uu\ee\dd\uu\dd & 53142 \\
15 & \ee\uu\uu\dd\dd\uu\dd & 51342 \\
16 & \ee\uu\uu\dd\dd\ee & 52341 \\
17 & \ee\uu\uu\ee\dd\dd & 52134 \\
18 & \ee\uu\uu\uu\dd\dd\dd & 51234 \\
19 & \ee\uu\uu\dd\ee\dd & 52314 \\
20 & \ee\uu\uu\dd\uu\dd\dd & 51324 \\
21 & \ee\uu\ee\uu\dd\dd & 53124 \\
22 & \ee\uu\ee\ee\dd & 53214 \\
%===================
23 & \uu\dd\uu\ee\ee\dd & 35214 \\
24 & \uu\dd\uu\ee\uu\dd\dd & 35124 \\
25 & \uu\dd\uu\uu\dd\uu\dd\dd & 15324 \\
26 & \uu\dd\uu\uu\dd\ee\dd & 25314 \\
27 & \uu\dd\uu\uu\uu\dd\dd\dd & 15234 \\
28 & \uu\dd\uu\uu\ee\dd\dd & 25134 \\
29 & \uu\dd\uu\uu\dd\dd\ee & 25341 \\
30 & \uu\dd\uu\uu\dd\dd\uu\dd & 15342 \\
\end{array} & \begin{array}[t]{ccc} n & \Sch_4(n) & \Phi_4(n) \\ \hline
31 & \uu\dd\uu\ee\dd\uu\dd & 35142 \\
32 & \uu\dd\uu\dd\ee\ee & 35241 \\
33 & \uu\dd\uu\dd\ee\ee & 35421 \\
34 & \uu\dd\uu\dd\ee\uu\dd & 35412 \\
35 & \uu\dd\uu\dd\uu\dd\uu\dd & 15432 \\
36 & \uu\dd\uu\dd\uu\dd\ee & 25431 \\
37 & \uu\dd\uu\dd\uu\uu\dd\dd & 15423 \\
38 & \uu\dd\uu\dd\uu\ee\dd & 25413 \\
39 & \uu\dd\ee\uu\ee\dd & 45213 \\
40 & \uu\dd\ee\uu\uu\dd\dd & 45123 \\
41 & \uu\dd\ee\uu\dd\ee & 45231 \\
42 & \uu\dd\ee\uu\dd\uu\dd & 45132 \\
43 & \uu\dd\ee\ee\uu\dd & 45312 \\
44 & \uu\dd\ee\ee\ee & 45321 \\
%===================
45 & \uu\uu\dd\dd\ee\ee &  34521 \\
46 & \uu\uu\dd\dd\ee\uu\dd & 34512\\
47 & \uu\uu\dd\dd\uu\dd\uu\dd & 14532\\
48 & \uu\uu\dd\dd\uu\dd\ee & 24531\\
49 & \uu\uu\dd\dd\uu\uu\dd\dd & 14523\\
50 & \uu\uu\dd\dd\uu\ee\dd & 24513\\
51 & \uu\ee\dd\uu\ee\dd & 42513\\
52 & \uu\ee\dd\uu\uu\dd\dd & 41523\\
53 & \uu\ee\dd\uu\dd\ee & 42531\\
54 & \uu\ee\dd\uu\dd\uu\dd & 41532\\
55 & \uu\ee\dd\ee\uu\dd & 43512\\
56 & \uu\ee\dd\ee\ee & 43521\\
%===================
57 & \uu\ee\ee\dd\ee & 43251\\
58 & \uu\ee\ee\dd\uu\dd & 43152\\
59 & \uu\ee\uu\dd\dd\uu\dd & 41352\\
60 & \uu\ee\uu\dd\dd\ee & 42351\\
\end{array} & \begin{array}[t]{ccc} n & \Sch_4(n) & \Phi_4(n) \\ \hline
61 & \uu\uu\dd\uu\dd\dd\ee & 24351\\
62 & \uu\uu\dd\uu\dd\dd\uu\dd & 14352\\
63 & \uu\uu\dd\ee\dd\uu\dd & 34152\\
64 & \uu\uu\dd\ee\dd\ee & 34251\\
65 & \uu\uu\uu\dd\dd\dd\ee & 23451\\
66 & \uu\uu\uu\dd\dd\dd\uu\dd & 13452\\
67 & \uu\uu\ee\dd\dd\uu\dd & 31452\\
68 & \uu\uu\ee\dd\dd\ee & 32451\\
%===================
69 & \uu\uu\ee\ee\dd\dd & 32145 \\
70 & \uu\uu\ee\uu\dd\dd\dd & 31245 \\
71 & \uu\uu\uu\dd\uu\dd\dd\dd & 13245 \\
72 & \uu\uu\uu\dd\ee\dd\dd & 23145 \\
73 & \uu\uu\uu\uu\dd\dd\dd\dd & 12345 \\
74 & \uu\uu\uu\ee\dd\dd\dd & 21345 \\
75 & \uu\uu\uu\dd\dd\ee\dd & 23415 \\
76 & \uu\uu\uu\dd\dd\uu\dd\dd & 13425 \\
77 & \uu\uu\ee\dd\uu\dd\dd & 31425 \\
78 & \uu\uu\dd\ee\ee\dd & 32415 \\
79 & \uu\uu\dd\ee\ee\dd & 34215 \\
80 & \uu\uu\dd\ee\uu\dd\dd & 34125 \\
81 & \uu\uu\dd\uu\dd\uu\dd\dd & 14325 \\
82 & \uu\uu\dd\uu\dd\ee\dd & 24315 \\
83 & \uu\uu\dd\uu\uu\dd\dd\dd & 14235 \\
84 & \uu\uu\dd\uu\ee\dd\dd & 24135 \\
85 & \uu\ee\uu\ee\dd\dd & 42135 \\
86 & \uu\ee\uu\uu\dd\dd\dd & 41235 \\
87 & \uu\ee\uu\dd\ee\dd & 42315 \\
88 & \uu\ee\uu\dd\uu\dd\dd & 41325 \\
89 & \uu\ee\ee\uu\dd\dd & 43125 \\
90 & \uu\ee\ee\ee\dd & 43215 \\
\end{array} \\ \hline
\end{array}$
}
\end{center}
\end{table}

\section{Regular patterns and Gray codes}

Here we present a general generating algorithm and Gray codes for
permutations avoiding a set of patterns $T$, provided
$T$ satisfies certain constraints.
The operations of reverse, complement and their
composition extend these to codes for $T^c$, $T^r$ and $T^{rc}$.
Our approach is based on
generating trees; see  \cite{BBGP,bernini,Chow_West,Wes_94} and the references therein.
In \cite{bernini} a general Gray code for a very large family of combinatorial
objects is given; objects are encoded by their corresponding path in the
generating tree and often it is possible to translate the
obtained codes into codes for objects.
The method we present here is, in a way,
complementary to that of \cite{bernini}: it
works for a large family of patterns and
objects are produced in `natural' representation. It
is also easily implemented by efficient generating algorithms.
Its disadvantage is, for example, that it gives a distance-5 Gray code for
$\s(231)$, and so is less optimal than the one given in Section 2;
and it does not work for $T=\{1243,2143\}$ (the set of patterns
considered in Section 3) since $T$ does not satisfy the required criteria.

We begin by explaining the generating trees technique in the context
of pattern avoidance. The {\it sites} of $\pi\in \s_n$ are the
positions between two consecutive entries, as well as before the first and
after the last entry; and they are numbered, from right to left,
from $1$ to $n+1$. For a  permutation $\pi\in \s_n(T)$, with $T$ a
set of forbidden patterns, $i$ is an {\it active site} if the
permutation obtained from $\pi$ by inserting $n+1$ into its $i$-th
site is a permutation in $\s_{n+1}(T)$; we call such a permutation
in $\s_{n+1}(T)$ a {\it son} of $\pi$.
Clearly, if $\pi\in \s_{n+1}(T)$,
by erasing $n+1$ in $\pi$ one obtains a permutation in $\s_n(T)$;
or equivalently, any permutation in $\s_{n+1}(T)$ is obtained from a permutation in
$\s_n(T)$ by inserting $n+1$ into one of its active sites.
The active sites of a
permutation $\pi\in \s_n(T)$ are {\it right justified} if the sites
to the right of any active site are also active. We denote by
$\chi_T(i,\pi)$ the number of active sites of the permutation
obtained from $\pi$ by inserting $n+1$ into its $i$-th active site.

A set of patterns $T$ is called {\it regular} if for any $n\geq 1$ and $\pi\in \s_n(T)$

\begin{itemize}
     \item[$\bullet$] $\pi$ has at least two active sites and they are right justified;
     \item[$\bullet$] $\chi_T(i,\pi)$ does not depend on $\pi$
    but only on the number $k$ of active sites of $\pi$; in this case
    we denote $\chi_T(i,\pi)$ by
    $\chi_T(i,k)$.
\end{itemize}

In what follows we shall assume that $T$ is a regular set of patterns.
Several examples of regular patterns $T$, together with their respective $\chi$ functions,
are given at the end of this section.

\medskip

Now we will describe an efficient (constant amortized time)
generating algorithm for permutations avoiding a regular set of
patterns; then we show how we can modify it to obtain Gray codes. If
$n=1$, then $\s_n(T)=\{(1)\}$; otherwise $\s_n(T)= \cup_{\pi\in
\s_{n-1}(T)}\{\sigma\in \s_n\,|\,\sigma {\rm \ is\ a\ son\ of\ } \pi
\}$. An efficient implementation is based on the following
considerations and its pseudocode is given in Algorithm 2. The
permutation obtained from $\pi\in \s_{n-1}(T)$ by inserting $n$ into
its first (rightmost) active site is $\pi n$. Let $\sigma$ (resp.
$\tau$) be the permutation obtained from $\pi$ by inserting $n$ into
the $i$-th (resp. $(i+1)$-th) active site of $\pi$. In this case
$\tau$ is obtained by transposing the entries in positions $n-i+1$
and $n-i$ of $\sigma$. In addition, if $\chi_T(i,k)$ is calculable,
from $i$ and $k$, in constant time, then the obtained algorithm,
Gen\_Avoid (Algorithm 2), runs in constant amortized time. Indeed,
this algorithm satisfies the following properties:
\begin{itemize}
\item the
total amount of computation in each call is proportional with the
number of direct calls produced by this call,
\item each non-terminal call produces at least two
     recursive calls (i.e., there is no call of degree one), and
\item each terminal call (degree-zero call) produces a new permutation,
\end{itemize}
see for instance \cite{Rus_00} and Figure \ref{fig_v} (a) for an example.

\begin{figure}
\caption{\label{fig_v}
(a) The generating tree induced by the call of Gen\_Avoid(1,2) for $n=4$
and with $\chi$ defined by: $\chi(1,k)=k+1$ and $\chi(i,k)=i$ if $i\neq 1$.
It corresponds to the forbidden pattern $T=\{321\}$. The active sites
are represented by a dot.
(b) The first four levels of the generating tree induced by the definition
(\ref{gen_C}) with the same function $\chi$;
they yield the lists $\C_{i}(321)$ for the sets
$\s_i(321)$, $1\leq i \leq 4$.
This tree is the Gray-code ordered version of the one in (a).
Permutations in bold have direction $down$ and the others direction $up$.
}
\begin{tabular}{cc}
\scalebox{0.8}{\includegraphics{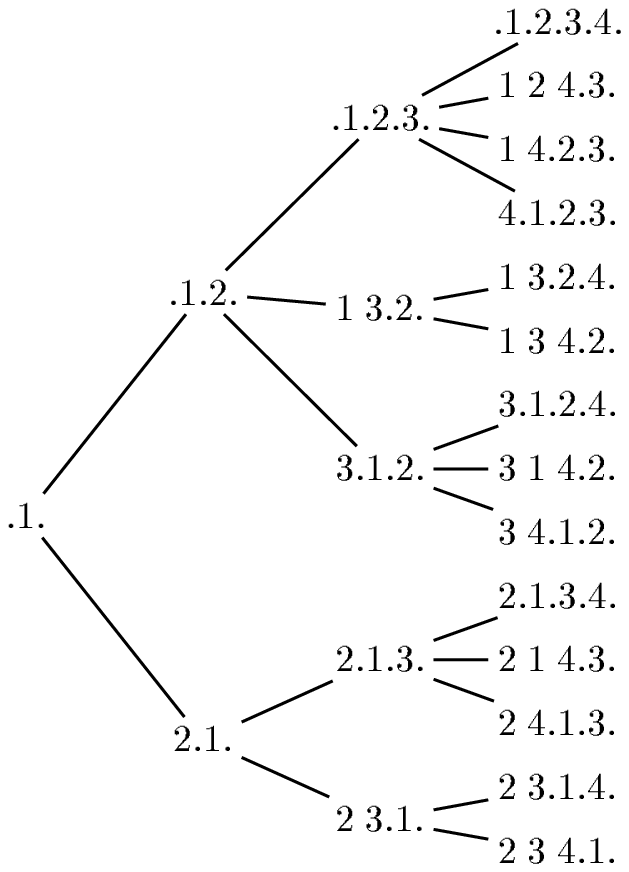}}&\scalebox{0.8}{\includegraphics{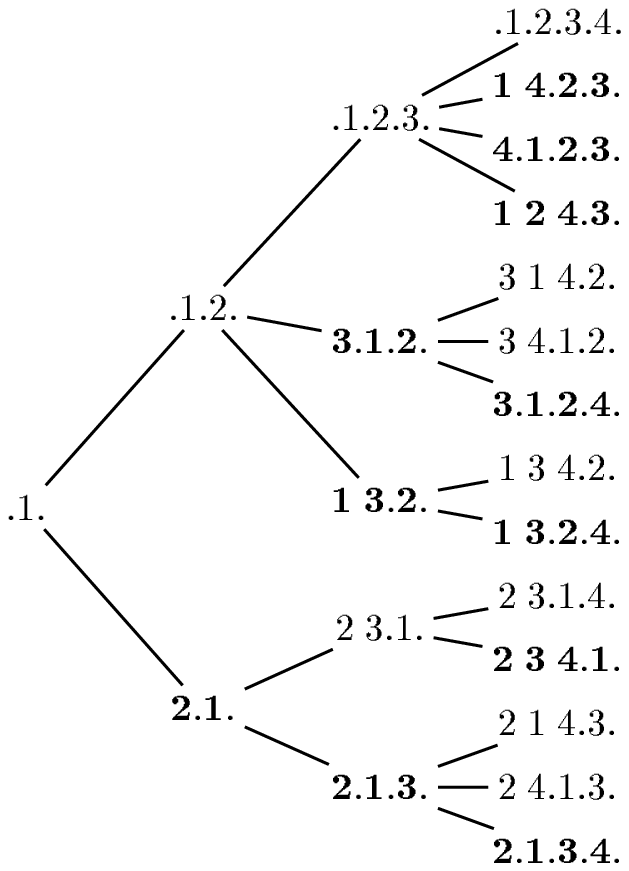}}\\
(a) & (b)
\end{tabular}
\end{figure}

Now we show how one can modify the generating procedure Gen\_Avoid sketched above in order to
produce  a Gray code listing.
We associate to each permutation  $\pi\in \s_n(T)$
\begin{itemize}
\item a {\it direction}, {\it up} or {\it down}, and we denote by $\pi^1$
     the permutation $\pi$ with  direction {\it up} and by $\pi^0$
     the permutation $\pi$ with direction {\it down}.
     A permutation together with its direction is called {\it directed permutation}.
\item a list of successors, each of them a permutation in  $\s_{n+1}(T)$.
     The first permutation in the list of successors of
     $\pi^1$ has direction {\it up} and all others have direction {\it down}.
     The list of successors of $\pi^0$ is obtained by reversing
     the list of successors of $\pi^1$ and then reversing the direction of each
     element of the list.
\end{itemize}

Let $\pi\in \s_n(T)$ with $k$ successors (or, equivalently, $k$ active sites),
and $L_k$ be the unimodal sequence of integers

\begin{equation}\label{ordered_grow}
L_k\;=\;
\left\{ \begin{array}{ll}
      1,3,5,\ldots, k,   (k-1),(k-3),\ldots,4,2   & {\rm if\ } k {\rm\ is\ odd}       \\
      1,3,5,\ldots,(k-1),k,(k-2),\ldots,4,2   & {\rm if\ } k {\rm\ is\ even}.
     \end{array}
     \right.
\end{equation}
This list is very important in our construction of a Gray code;
it has the following critical properties, independent of $k$:
it begins and ends with the same element, and the difference between
two consecutive elements is less than or equal to $2$.

For  a permutation $\pi$ with $k$ active sites,
the list of successors of $\pi^1$, denoted by $\phi(\pi^1)$, is a
list of $k$ directed permutations in $\s_{n+1}(T)$: its $j$-th
element is obtained from $\pi$ by inserting $n+1$ in the $L_k(j)$-th
active site of $\pi$; and as stated above, the first permutation in $\phi(\pi^1)$ has
direction {\it up} and all others have direction {\it down}.
And we extend $\phi$ in natural way to lists of directed permutations:
$\phi(\pi(1),\pi(2),\ldots)$ is simply the list $\phi(\pi(1)),\phi(\pi(2)),\ldots$.
This kind of distribution of directions among the successors of an object
is similar to that of \cite{Weston_Vaj}.

Let $d_n={\rm card}(\s_n(T))$ and define the list
\begin{equation}\label{gen_C}
\mathcal C_n(T)=
\mathcal C_n= \bigoplus_{q=1}^{d_{n-1}}\phi({\mathcal C}_{n-1}(q))
\end{equation}
where ${\mathcal C}_n(q)$ is the $q$-th directed permutation of
${\mathcal C}_n$, anchored by ${\mathcal C}_{1}= (1)^1$. We will
show that the list of permutations in ${\mathcal C}_n$ (regardless
of their directions) is a Gray code with distance $5$ for the set
$\s_{n}(T)$. With these considerations in mind we have

\begin{lemma}$ $
\begin{itemize}
\item
The list ${\mathcal C}_{n}$ contains all $T $-avoiding
permutations exactly once;
\item
The first permutation in ${\mathcal C}_{n}$ is $(1,\ldots,n)$ and
the last one is $(2,1,3,\ldots,n)$.
\end{itemize}
\end{lemma}

\begin{lemma}
\label{direct}
If $\pi^i$ is a directed permutation in ${\mathcal C}_{n}$
(that is, $\pi$ is a length $n$ permutation and $i\in\{0,1\}$ is a direction), then
two successive permutations in $\phi(\pi^i)$,
say $\sigma$ and $\tau$, differ in at most three positions.
\end{lemma}
\begin{proof}
Since $\phi(\pi^0)$ is the reverse of $\phi(\pi^1)$ it is enough to
prove the statement for $i=1$; so suppose that $i=1$. Let
$\sigma$ and $\tau$ be the permutations obtained by inserting $n+1$
in the $L_k(j)$-th and $L_k(j+1)$-th active site of $\pi$, respectively,
for some $j$. Since $|L_k(j)-L_k(j+1)|\leq 2$, $d(\sigma,\tau)\leq 3$.
\end{proof}

Let $\pi^i\in {\mathcal C}_{n}$ and $\ell(\pi^i)$ denote the first (leftmost)
element of the list $\phi(\pi^i)$, $\ell^2(\pi^i)=\ell(\ell(\pi^i))$,
and $\ell^s(\pi^i)=\ell(\ell^{s-1}(\pi^i))$.
Similarly, $r(\pi^i)$ denotes the last (rightmost)
element of the list $\phi(\pi^i)$, and $r^s(\pi^i)$ is defined analogously.
For $\pi^i\in {\mathcal C}_{n}$ let ${\rm dir}(\pi^i)=i\in\{0,1\}$.
%Obviously,
%$\ell(\pi^0)=r(\pi^1)$, $\ell(\pi^1)=r(\pi^0)$.
By the recursive application of the definition of the list $\phi(\pi^i)$ we have
the following lemma.

\begin{lemma}\label{puis}
If $\pi^i\in {\mathcal C}_{n}$,
then ${\rm dir}(\ell^s(\pi^i))=1$ and ${\rm dir}(r^s(\pi^i))=0$ for any $s\geq 1$.
\end{lemma}
\begin{proof}
$\ell(\pi^i)$, the first successor of $\pi^i$ has direction $up$ for any $i\in\{0,1\}$,
and generally ${\rm dir}(\ell^s(\pi^i))=1$ for $s\geq 1$.
Similarly, $r(\pi^i)$, the last successor of $\pi^i$ has direction $down$ for any $i\in\{0,1\}$,
and ${\rm dir}(r^s(\pi^i))=0$ for $s\geq 1$.
\end{proof}

\begin{table}\label{ttable_321}
\begin{center}
\caption{The Gray code list $\C_5(321)$ for the set $\s_5(321)$
given by relation (\ref{gen_C}) and with succession function $\chi$ in Paragraph
\ref{paragraph}. Permutations are listed column-wise in 14 groups;
each group contains the sons of a same permutation
in $\s_4(321)$, see Figure \ref{fig_v} b.
In bold are permutations with direction $down$ and the others
with direction $up$.}

{\small
$\begin{array}{|c|c|c|c|c|}
\hline
\begin{array}[t]{c}
12345 \\
\mathbf{12534}\\
\mathbf{51234} \\
\mathbf{15234} \\
\mathbf{12354} \\ \hline
%%%%%%%%%%%%%%%%%%%%%%
14253 \\
14523\\
\mathbf{14235} \\ \hline
%%%%%%%%%%%%%%%%%%%%%%
41253 \\
\end{array}
&
\begin{array}[t]{c}
45123 \\
41523 \\
\mathbf{41235}\\ \hline
%%%%%%%%%%%%%%%%%%%%%%
12453\\
\mathbf{12435}\\ \hline
%%%%%%%%%%%%%%%%%%%%%%
31425\\
\mathbf{31452}\\ \hline
%%%%%%%%%%%%%%%%%%%%%%
34125\\
\mathbf{34512}\\
\end{array}&

\begin{array}[t]{c}
\mathbf{34152} \\ \hline
%%%%%%%%%%%%%%%%%%%%%%
31254 \\
35124 \\
31524 \\
\mathbf{31245} \\ \hline
%%%%%%%%%%%%%%%%%%%%%%
13425 \\
\mathbf{13452} \\ \hline
%%%%%%%%%%%%%%%%%%%%%%
13254 \\
13524 \\
\end{array}&
\begin{array}[t]{c}
\mathbf{13245} \\ \hline
%%%%%%%%%%%%%%%%%%%%%%
23145 \\
\mathbf{23514} \\
\mathbf{23154} \\ \hline
%%%%%%%%%%%%%%%%%%%%%%
23451 \\
\mathbf{23415} \\ \hline
%%%%%%%%%%%%%%%%%%%%%%
21435 \\
\mathbf{21453} \\ \hline
%%%%%%%%%%%%%%%%%%%%%%
24135 \\
\end{array}&
\begin{array}[t]{c}
\mathbf{24513}  \\
\mathbf{24153}  \\ \hline
%%%%%%%%%%%%%%%%%%%%%%
21354  \\
25134  \\
21534  \\
\mathbf{21345}  \\

\end{array}
\\ \hline
\end{array}$
}
\end{center}
\end{table}

\begin{lemma}\label{min2}
If $\sigma,\tau\in \s_n(T)$ and $d(\sigma,\tau)\leq p$,
then, for $s\geq 1$,
$$d(r^s(\sigma^0),\ell^s(\tau^1))\leq p.$$
\end{lemma}
\begin{proof}
$r(\sigma^0)=(\sigma, (n+1))^0$ and $\ell(\tau^1)=(\tau, (n+1))^1$.
Induction on $s$ completes the proof.
\end{proof}

\begin{theorem}
Two consecutive permutations in ${\mathcal C}_{n}$
differ in at most five positions.
\end{theorem}
\begin{proof}
Let $\sigma^i$ and $\tau^j$ be two consecutive elements of ${\mathcal C}_{n}$.
If there is a $\pi^m\in{\mathcal C}_{n-1}$ such that $\sigma^i,\tau^j\in\phi(\pi^m)$,
then, by Lemma \ref{direct}$, \sigma$ and $\tau$ differ in at most three positions.
Otherwise, let $\pi^m$ be the closest common ancestor of $\sigma^i$ and $\tau^j$ in the generating tree,
that is, $\pi$ is the longest permutation such that there exists a direction
$m\in\{0,1\}$ with $\sigma^i,\tau^j\in \phi(\phi(\ldots\phi(\pi^m)\ldots))$.
In this case, there exist $\alpha^a$ and $\beta^b$ successive elements
in $\phi(\pi^m)$ (so that $\alpha$ and $\beta$ differ in at most three positions)
and an $s\geq 1$ such that $\sigma^i=r^s(\alpha^a)$ and  $\tau^j=\ell^s(\beta^b)$.

If $s=1$, then $\sigma$ and $\tau$ are obtained from $\alpha$ and $\beta$
by the insertion of their largest element in the first or second active site,
according to $a$ and $b$;
in these cases $\sigma$ and $\tau$ differ in at most five positions.
(Actually, if $a=b$, then $\sigma$ and $\tau$ differ as  $\alpha$ and $\beta$,
that is, in at most three positions.)

If $s>1$, by Lemma \ref{puis},
${\rm dir}(r(\alpha^a))=\ldots={\rm dir}(r^s(\alpha^a))=0$ and
${\rm dir}(\ell(\beta^b))=\ldots={\rm dir}(\ell^s(\beta^b))=1$.
Since $r(\alpha^a)$ and $\ell(\beta^b)$ differ in at most five positions,
by Lemma \ref{min2}, so are $\sigma$ and $\tau$.
\end{proof}

The first and last permutations in ${\mathcal C}_{n}$ have distance
two, so ${\mathcal C}_{n}$ is a circular Gray code, see Table 4. The
generating algorithm Gen\_Avoid sketched in the beginning of this
section and presented in Algorithm 2 can be easily modified to
generate the list ${\mathcal C}_n(T)$ for any set of regular
patterns: it is enough to change appropriately the order among its
successive recursive calls by endowing each permutation with a
direction as described above; see also Figure \ref{fig_v}.

\begin{figure}
\begin{tabular}{l}
  \hline
  {\bf Algorithm 2} Pseudocode for generating permutations avoiding a set $T$ of regular\\
  patterns characterized by the succession function $\chi(i,k)$. After the initialization of\\
  $\pi$ by the length 1 permutation $[1]$, the call of Gen\_Avoid($1,2$) produces $\s_n(T)$. Its\\
  ordered version, as described in Section 4, produces distance-5 Gray codes.\\\hline
 {\bf procedure} Gen\_Avoid($size,k$)\\
 $\quad${\bf if} $size=n$ {\bf then}\\
 $\qquad$Print($\pi$)\\
 $\quad${\bf else}\\
 $\qquad size:=size+1$\\
 $\qquad\pi:=[\pi, size]$\\
 $\qquad$Gen\_Avoid($size,\chi(1,k)$)\\
 $\qquad${\bf for} $i:=1$ {\bf to} $k-1$ {\bf do}\\
 $\qquad\quad\pi:= (size-i+1,size-i)\circ \pi$\\
 $\qquad\quad$Gen\_Avoid($size,\chi(i+1,k)$)\\
 $\qquad${\bf end for}\\
 $\qquad${\bf for} $i:=k-1$ {\bf to} $1$ {\bf by} $-1$ {\bf do}\\
 $\qquad\quad\pi:= (size-i+1,size-i)\circ \pi$\\
 $\qquad${\bf end for}\\
 $\quad${\bf end if}\\
 {\bf end procedure}\\ \hline
\end{tabular}
\end{figure}

\subsection{Several well-known classes of regular patterns}\label{paragraph}
Below we give several classes of regular patterns together with the $\chi$ function.
For each class, a recursive construction is given in the corresponding reference(s);
it is based (often implicitly) on the distribution of active sites of the permutations
belonging to the class. It is routine to express these recursive constructions
in terms of $\chi$ functions and check the regularity of each class.

\noindent
Classes given by counting sequences:
\begin{enumerate}
\item[(i)]
$2^{n-1}$ \cite{BarBerPon}.\\
$T=\{321,312\}$,
$\chi_T(i,k)=2$

\item[(ii)]
Pell numbers \cite{BarBerPon}.\\
$T=\{321,3412,4123\}$,
     $\chi_T(i,k)=
     \left\{ \begin{array}{ll}
        3 & {\rm if}\ i=1       \\
        2  & {\rm otherwise}
     \end{array}
     \right.$

\item[(iii)]
even-index Fibonacci numbers \cite{BarBerPon}.
\begin{itemize}
     \item[-] $T=\{321,3412\}$,
     $\chi_T(i,k)=
     \left\{ \begin{array}{ll}
        k+1 & {\rm if}\ i=1       \\
        2  & {\rm otherwise}
     \end{array}
     \right.$
     \item[-] $T=\{321,4123\}$,
     $\chi_T(i,k)=
     \left\{ \begin{array}{ll}
        3 & {\rm if}\ i=1       \\
        i  & {\rm otherwise}
     \end{array}
     \right.$
\end{itemize}

\item[(iv)] \label{for_cat}
     Catalan numbers \cite{Pall_88,Wes_94}.
     \begin{itemize}
     \item[-] $T=\{312\}$,
     $\chi_T(i,k)=i+1$
     \item[-]  $T=\{321\}$,
     $\chi_T(i,k)=
     \left\{ \begin{array}{ll}
        k+1 & {\rm if}\ i=1       \\
        i  & {\rm otherwise}
     \end{array}
     \right.$

     \end{itemize}
\item[(v)] Schr\"oder numbers \cite{Guibert}.
     \begin{itemize}
     \item[-] $T=\{4321,4312\}$,
     $\chi_T(i,k)=
     \left\{ \begin{array}{ll}
       k+1 & {\rm if}\ i=1 {\rm\ or\ }i=2       \\
       i   & {\rm otherwise}
       \end{array}
     \right.$
     \item[-] $T=\{4231,4132\}$,
     $\chi_T(i,k)=
     \left\{ \begin{array}{ll}
       k+1 & {\rm if}\ i=1 {\rm\ or\ }i=k       \\
       i+1   & {\rm otherwise}
       \end{array}
       \right.$
     \item[-] $T=\{4123,4213\}$,
     $\chi_T(i,k)=
     \left\{ \begin{array}{ll}
       k+1 & {\rm if}\ i=k-1 {\rm\ or\ }i=k       \\
       i+2   & {\rm otherwise}
       \end{array}
       \right.$

     \end{itemize}
\item[(vi)] central binomial coefficients $\binom{2n-2}{n-1}$ \cite{Guibert}.
     \begin{itemize}
     \item[-] $T=\{4321,4231,4312,4132\}$,
     $\chi_T(i,k)=
     \left\{ \begin{array}{ll}
       k+1 & {\rm if}\ i=1       \\
       3   & {\rm if}\ i=2       \\
       i   & {\rm otherwise}
       \end{array}
     \right.$

     \item[-] $T=\{4231,4132,4213,4123\}$,
     $\chi_T(i,k)=
     \left\{ \begin{array}{ll}
       3   & {\rm if}\ i=1       \\
       i+1 & {\rm otherwise}
       \end{array}
       \right.$
     \end{itemize}

\end{enumerate}

\noindent
Variable length patterns:

\begin{enumerate}
\item[(a)]
     $T=\{321,(p+1)12\ldots p\}$,
     $\chi_T(i,k)=
     \left\{ \begin{array}{ll}
        k+1 & {\rm if}\ i=1 {\rm \ and\ } k<p        \\
        p & {\rm if}\ i=1 {\rm \ and\ } k=p       \\
        i  & {\rm otherwise}
     \end{array}
     \right.$

\noindent
See for instance \cite{Chow_West,BarBerPon}. If $p=2$, then we retrieve the case (i) above;
$p=3$ corresponds to $T=\{321,4123\}$ in case (iii); and
$p=\infty$ corresponds to $T=\{321\}$ in case (iv).

\item[(b)]
     $T=\{321,3412,(p+1)12\ldots p\}$,
     $\chi_T(i,k)=
     \left\{ \begin{array}{ll}
        k+1 & {\rm if}\ i=1 {\rm \ and\ } k<p        \\
        p & {\rm if}\ i=1 {\rm \ and\ } k=p       \\
        2  & {\rm otherwise}
     \end{array}
     \right.$

\noindent
See for instance \cite{BarBerPon}. If $p=2$, then we retrieve the case (i) above;
if $p=3$, the case (ii);
and $p=\infty$ corresponds to $T=\{321,3412\}$ in case (iii).

\item[(c)]
     $T=\cup_{\tau\in\s_{p-1}}\{(p+1)\tau p\}$.\\
     $\chi_T(i,k)=
     \left\{ \begin{array}{ll}
        k+1 & {\rm if}\  k<p  {\rm \ or\ }  i>k-p+1    \\
        i+p-1  & {\rm otherwise}
     \end{array}
     \right.$

\noindent
See \cite{BdLPP_00,Kremer_2000,Kremer_2003}.
If $p=2$, then we retrieve the case $T=\{312\}$ in point (iv) above;
and $p=3$ corresponds to  $T=\{4123,4213\}$ in point (v).
\end{enumerate}

\section*{Acknowledgments}
The authors kindly thank the anonymous referees for their
helpful suggestions which have greatly improved the accuracy and presentation of this work.
The first two authors would also like to thank Toast, Dublin, for their hospitality during the preparation of this document.


\begin{thebibliography}{99}
\bibitem{BBGP} Silvia Bacchelli, Elena Barcucci, Elisabetta Grazzini and Elisa Pergola,
Exhaustive generation of combinatorial objects by ECO, {\em Acta
Informatica} {\bf 40}:8 (2004) 585-602.

\bibitem{baril} Jean-Luc Baril, Gray code for permutations with a fixed number of cycles,
{\em Disc. Math.}  {\bf30}:13 (2007) 1559--1571.

\bibitem{BdLPP_00} Elena Barcucci, Alberto Del Lungo, Elisa Pergola and Renzo Pinzani,
Permutations avoiding an increasing number of length-increasing
forbidden subsequences, {\em Disc. Math. Theor. Comp. Sci.}
{\bf 4}:1 (2000) 31--44.

\bibitem{BarBerPon} Elena Barcucci, Antonio Bernini and Maddalena Poneti,
From Fibonacci to Catalan permutations, {\em PuMA} {\bf17}(1-2)
(2006) 1--17.

\bibitem{baril_vaj} Jean-Luc Baril and Vincent Vajnovszki, Gray code for derangements, {\em Disc. App. Math.}
{\bf 140} (2004) 207--221.

\bibitem{bernini} Antonio Bernini, Elisabetta Grazzini, Elisa Pergola and Renzo Pinzani,
A general exhaustive generation algorithm for Gray structures,
{\em Acta Informatica} {\bf 44}:5 (2007) 361--376. Also as
{\em preprint math.CO/0703262}.

\bibitem{Chow_West} Timothy Chow and  Julian West,
Forbidden sequences and {C}hebyshev polynomials, {\em Disc. Math.}
{\bf 204} (1999) 119--128.

\bibitem{bonabook} Mikl\'{o}s B\'{o}na. Combinatorics of Permutations. Chapman \& Hall, 2004.

\bibitem{egge_mansour} Eric S. Egge and Toufik Mansour, Permutations which
avoid $1243$ and $2143$, continued fractions, and Chebyshev
polynomials, {\em Elec. J. Comb.} {\bf 9}:2 (2003) \#R6.

\bibitem{Ehr_73} Gideon Ehrlich,
Loopless algorithms for generating permutations, combinations, and
other combinatorial objects, {\em J. ACM} {\bf 20} (1973) 500--513.

\bibitem{Guibert}
Olivier Guibert, Combinatoire des permutations \`{a} motifs exclus
en liaison avec mots, cartes planaires et tableaux de Young,  PhD
thesis, Universit\'{e} Bordeaux 1, 1995.

\bibitem{juarna_vaj} Asep Juarna and Vincent Vajnovszki,
Some generalizations of a Simion-Schmidt bijection, {\em The
Computer Journal} {\bf50} (2007) 574--580.

\bibitem{Ko_92} C.W. Ko and Frank Ruskey,
Generating permutations of a bag by interchanges, {\em IPL} {\bf
41}:5 (1992) 263--269.

\bibitem{Kor_01}
James F. Korsh, Loopless generation of up-down permutations, {\em
Disc. Math.} {\bf 240}:1-3 (2001) 97--122.

\bibitem{Kremer_2000} Darla Kremer,
Permutations with forbidden subsequences and a generalized Schr{\"
o}der number, {\em Disc. Math.} {\bf 218}:1-3 (2000) 121--130.

\bibitem{Kremer_2003}
Darla  Kremer, Postscript: ``Permutations with forbidden
subsequences and a generalized Schr\"oder number" [Disc. Math. 218
(2000) 121-130]. {\em Disc. Math.} {\bf 270}:1-3 (2003) 332--333.

\bibitem{Pall_88}
Jean Pallo,
Some properties of the rotation lattice of binary trees,
{\em The Computer Journal} {\bf 31} (1988) 564--565.

\bibitem{Roe_92} Dominique {Roelants van Baronaigien} and Frank  Ruskey,
Generating permutations with given ups and downs, {\em Disc. Appl.
Math.} {\bf 36}:1 (1992) 57--65.

\bibitem{Rus_00} Frank Ruskey, Combinatorial Generation, book in preparation.
\bibitem{Rus_93} Frank Ruskey, Simple combinatorial {G}ray codes constructed by reversing sublists,
in {\em ISAAC Conference, LNCS} {\bf 762} (1993) 201--208.

\bibitem{savage} Carla Savage, A Survey of Combinatorial Gray Codes, {\em SIAM Rev.} {\bf 39}:4 (1997) 605--629.

\bibitem{Vaj_02_1} Vincent  Vajnovszki, {G}ray visiting {M}otzkins,
{\em Acta Informatica} {\bf 38} (2002) 793-811.

\bibitem{Vaj_02_2}
Vincent Vajnovszki, A loopless algorithm for generating the
permutations of a multiset, {\em Theor. Comp. Sci.} {\bf 307} (2003)
415-431.

\bibitem{Wal_01} Timothy Walsh, {G}ray codes for involutions,
{\em J. Combin. Math. Combin. Comput.} {\bf 36} (2001) 95--118.

\bibitem{Wes_94} Julian West,
Generating trees and the {C}atalan and {S}chr\"oder numbers, {\em
Disc. Math.} {\bf 146} (1994) 247--262.

\bibitem{Weston_Vaj}
Mark Weston and Vincent Vajnovszki, Gray codes for necklaces and
Lyndon words of arbitrary base, {\em PuMA} {\bf17}(1-2) (2006)
175--182.
\end{thebibliography}
\end{document}